\documentclass[a4paper, reqno]{amsart}
\usepackage{amsmath, amssymb, eucal, amscd, amstext, enumerate, mathrsfs, yfonts, amsfonts, comment, color}

\newtheorem{thm}{Theorem}
\newtheorem{lem}[thm]{Lemma}

\newtheorem{prop}[thm]{Proposition}
\newtheorem{cor}[thm]{Corollary}

\theoremstyle{definition}
\newtheorem{defn}[thm]{Definition}

\theoremstyle{remark}

\newtheorem{rem}[thm]{Remark}

\newtheorem{eg}[thm]{Example}

\numberwithin{equation}{section}

\newcommand{\ti}{\tilde}

\newcommand{\smnoind}{\smallskip\noindent}

\newcommand{\RP}{\mathbb{R}^+}

\newcommand{\BN}{\mathbb{N}}
\newcommand{\BR}{\mathbb{R}}

\newcommand{\CS}{\mathcal{S}}
\newcommand{\CA}{\mathcal{A}}
\newcommand{\CF}{\mathcal{F}}

\newcommand{\CP}{\mathcal{P}}
\newcommand{\CC}{\mathcal{C}}

\newcommand{\D}{\mathcal{D}}
\newcommand{\CE}{\mathcal{E}}

\newcommand{\CT}{\mathcal{T}}
\newcommand{\CI}{\mathcal{I}}
\newcommand{\CJ}{\mathcal{J}}
\newcommand{\CU}{\mathcal{U}}

\newcommand{\CB}{\mathcal{B}}

\newcommand{\KI}{\mathfrak{I}}

\newcommand{\FD}{\mathbf{D}}
\newcommand{\fd}{\mathbf{d}}

\begin{document}

\title[]{Coarse metric and uniform metric}

\author{Chi-Keung Ng}

\address[Chi-Keung Ng]{Chern Institute of Mathematics and LPMC, Nankai University, Tianjin 300071, China.}
\email{ckng@nankai.edu.cn}

\date{\today}



\begin{abstract}
We introduce the notion of coarse metric. 
Every coarse metric induces a coarse structure on the underlying set. 
Conversely, we observe that all coarse spaces come from a particular type of coarse metric in a unique way. 
In the case when the coarse structure $\mathcal{E}$ on a set $X$ is defined by a coarse metric that takes values in a meet-complete totally ordered set, we define the associated Hausdorff coarse metric on the set $\mathcal{P}_0(X)$ of non-empty subsets of $X$ and show that it induces the Hausdorff coarse structure on $\mathcal{P}_0(X)$.

On the other hand, we define the notion of pseudo uniform metric. 
Each pseudo uniform metric induces a uniform structure on the underlying space. 
In the reverse direction, we show that  a uniform structure $\mathcal{U}$ on a set $X$ is induced by a map $d$ from $X\times X$ to a partially ordered set (with no requirement on $d$) if and only if $\mathcal{U}$ admits a base $\mathcal{B}$ such that $\mathcal{B}\cup \{\bigcap \mathcal{U}\}$ is closed under arbitrary intersections.
In this case, $\mathcal{U}$ is actually defined by a pseudo uniform metric. 
We also show that a uniform structures $\mathcal{U}$ comes from a pseudo uniform metric that takes values in a totally ordered set if and only if $\mathcal{U}$ admits a totally ordered base. 

Finally, a valuation ring will produce an example of a coarse and pseudo uniform metric that take values in a totally ordered set.
\end{abstract}

\maketitle

\section{Introduction and Notations}

The notion of coarse spaces was first introduced in \cite{HPR}.
It can be regarded as an abstract framework for the study of large scale properties of metric spaces. 
A throughout account for coarse spaces can be found in \cite{Roe} (see also \cite{BLS, CWY, DZ, HR, HRY, MY, NR, Wr03, Wr11, Yam, Zhang} for some information on coarse structure). 
On the other hand, coarse structure can be deemed as an opposite to uniform structure (see e.g. \cite{Isb} or \cite{Page}). 
Obviously, pseudo metric is a common source of examples for both of them. 
In fact, uniform spaces are generalizaions of pseudo metric spaces and coarse structures were first studied for metric spaces (see \cite{HPR}). 
It is known that a coarse structure (respectively, a uniform structure) is defined by a pseudo metric if and only if it has a countable base (see e.g. \cite[Theorem 2.55]{Roe} for the case of coarse structures and \cite[Corollary I.4.4]{Page} for the case of uniform structures). 
The aim of this article is to introduce two general notions of metric, and to study coarse structures and uniform structures associated respectively with them. 

\medskip

For a partially ordered set $\CI$ with a smallest element $0_\CI$, we denote by $\CI_\infty$ the extension of $\CI$ by adjoining a new element $\infty$ that is larger than all elements in $\CI$. 
The most general form of ``metric'' on a set $X$ is simply a map 
$$\fd:X\times X\to \CI_\infty.$$
If the map $\fd$ satisfies $\fd(x,y) = \fd(y,x)$ as well as $\fd(x,x)  = 0_\CI$ ($x,y\in X$), and it also fulfills certain growth condition as in Definition \ref{defn:coarse-met}(a) (respectively, descent condition as in Definition \ref{defn:unif-met}(b)), then $\fd$ is called a  coarse metric (respectively, uniform metric). 	
It will be obvious that a coarse metric (respectively, uniform metric) will induce a coarse structure (respectively, uniform structure) on the underlying space. 
A natural question is how to characterize those coarse structures and uniform structures coming from such generalized notions of metric. 

\medskip

In fact, it is not hard to check that any coarse space actually comes from a coarse metric. 
Furthermore, there is a bijective correspondence between the collection of coarse structures on a set $X$ and the collection of coarse metrics on $X$ that are ``saturated''  (Theorem \ref{thm:coarse metric}). 
Through the correspondence of coarse structures and coarse metrics, one can rephrase some terminologies in coarse spaces back in metric terms, which makes them easier to understand, and hopefully easier to manipulate.
For example, a coarse space is coarsely connected if and only if one (and hence all) of its defining coarse metrics does not take the value $\infty$.
A list of other translations can be found in Propositions \ref{prop:coarse-prop-met} and \ref{prop:bdd-geom}. 
Philosophically, the above correspondence tells us that by studying the coarse structure of a metric space, one actually ``forgets'' the triangle inequality and ``remembers'' only the growth condition (as in Definition \ref{defn:coarse-met}(a)). 

\medskip

On the other hand, we will show that a coarse structure has a totally ordered base if and only if it is defined by a coarse metric taking values in a totally ordered set (Corollary \ref{cor:Haus-met}(a)). 
Furthermore, we investigate the relation between the Hausdorff coarse structure on the set of non-empty subsets of a coarse space (Definition \ref{defn:Hausd-coar-str}) and the Hausdorff coarse metric on the same collection of subsets induced by the coarse metric defining the original coarse space (Proposition \ref{prop:Hausd-coar-met}). 

\medskip

In the case of uniform structures, the correspondence is not as perfect. 
We will show in Section 3 that for a uniform structure $\CU$, there is a map $\fd:X\times X\to \CI_\infty$  (without any further requirement on $\fd$) such that 
$$\big\{\FD_\alpha: \alpha\in \CI\setminus \{0_\CI\} \big\}$$
forms a base for $\CU$, where
\begin{equation}\label{eqt:defn-D-alpha}
\FD_\alpha:= \{(x,y): \fd(x,y)\leq \alpha \},
\end{equation}
if and only if $\CU$ admits a base $\CB$ with $\CB \cup \{\bigcap \CU \}$ being closed under arbitrary intersections.  
In this case, $\CU$ is actually defined by a pseduo uniform metric (Theorem \ref{thm:unif-str=>met}). 
In particular, if $\CU$ admits a totally ordered base, then it is defined by a pseduo uniform metric (Corollary \ref{cor:unif-tot-ord}).

\medskip

In Section 4, we give a mild condition, under which a coarse metric will become a pseudo uniform metric (Proposition \ref{prop:coarse>unif}). 
We will close this short article by giving an example of a coarse and pseudo uniform metric coming from a valuation ring (Example \ref{eg:val-ring}).

\medskip

In the remainder of this section, let us set some notation. 
Suppose that $\CI$ and $\CI_\infty$ are as in the second paragraph of this Introduction.
\begin{itemize}
	\item The smallest element in $\CI$, if it exists, is called the \emph{zero} of $\CI$ and will always be denoted by $0_\CI$. 
	
	\item $\CI$ is called an \emph{upward directed set} if for any $\alpha, \beta\in \CI$, there exists $\gamma \in \CI$ with $\alpha\leq \gamma$ and $\beta\leq \gamma$. 
	
	\item If $\CJ$ is another partially ordered set, and $\Lambda: \CI \to \CJ$ is a map, then 
	\begin{equation}\label{eqt:def-Lambda-infty}
	\text{we set $\Lambda_\infty: \CI_\infty\to \CJ_\infty$ to be the extension of $\Lambda$ such that $\Lambda_\infty(\infty) = \infty$.} 
	\end{equation}
\end{itemize}

\medskip

A partially ordered set $\CI$ is said to be \emph{meet-complete} if every non-empty subset $\CS$ of $\CI$ admits a greatest lower-bound $\inf \CS$ in $\CI$. 
Furthermore, $\CI$ is called a \emph{complete lattice} if it is meet-complete and each non-empty subset of $\CI$ has a least upper-bound in $\CI$. 

\medskip

For a set $X$, we use $\CP_0(X)$ to denote the collection of all non-empty subsets of $X$ and use $\Delta_X$ to denote the diagonal in $X\times X$; namely,
$$\Delta_X := \{(x,x)\in X\times X:x\in X \}.$$ 
Moreover, for any $A,B\subseteq X\times X$, we put $A^{-1}:=\{(y,x):(x,y)\in A \}$ as well as
$$A\circ B:= \big\{(x,z)\in X\times X: \text{ there exists }y\in X \text{ such that }(x,y)\in A
\text{ and }(y,z)\in B \big\}.$$
In this case, $A^{-1}$ is called the \emph{inverse} of $A$ and $A\circ B$ is called the \emph{product} of $A$ and $B$. 
Furthermore, if $S\subseteq X$, we denote 
$$B[S] := \{y\in X: (x,y)\in B, \text{ for some }x\in S\},$$ 
and $B[x] := B[\{x\}]$. 

\medskip

\begin{defn}\label{defn:gen-met}
Let $\CI$ be a partially ordered set with a zero, and $\fd:X\times X\to \CI_\infty$ be a map. 

\smnoind
(a) We set 
$$\FD(z,\alpha) := \FD_\alpha[z] \qquad (z\in X;\alpha\in \CI)$$
(where $\FD_\alpha$ is as in \eqref{eqt:defn-D-alpha}), and call $\FD(z,\alpha)$ the \emph{ball of radius $\alpha$ with center $z$}.
Moreover, we denote 
\begin{equation*}
\CB_\fd:= \big\{\FD_\alpha: \alpha\in \CI\setminus\{0_\CI\} \big\}.
\end{equation*}

\smnoind
(b) Suppose that $\fd$ satisfies   
	\begin{enumerate}[\ \ \ \ 	(D1)]
		\item $\fd(x,x) = 0_\CI$ for any $x\in X$;
		
		\item $\fd(x,y) = \fd(y,x)$ for any $x,y\in X$. 
	\end{enumerate}
Then $\fd$ is called a \emph{semi-$\CI$-metric}. 
\end{defn}

\medskip

\section{Coarse $\CI$-metric}

\medskip

In this section, we define and study coarse metrics and relate them to coarse structures. 
Recall that if $\CE\subseteq \CP_0(X\times X)$ such that $\Delta_X\in \CE$, and $\CE$ is closed under the formation of subsets, inverses, products and finite unions, then $\CE$ is called a \emph{coarse structure} on $X$. 
In this case, $(X,\CE)$ (or simply $X$) is called a \emph{coarse space}. 
A subcollection $\CB\subseteq \CE$ is called a base for $\CE$ if every element in $\CE$ is contained in an element of $\CB$. 

\medskip

\begin{defn}\label{defn:coarse-met}
Let $\CI$ be a upward directed set with a zero, and $\fd:X\times X\to \CI_\infty$ be a semi-$\CI$-metric 

\smnoind
(a) Suppose that there is a function $\Phi: \CI\to \CI$ such that $\fd(x,z)\leq \Phi(\alpha)$ whenever $\alpha\in \CI$ and $x,y,z\in X$ satisfy $\fd(x,y)\leq \alpha$ and $\fd(y,z)\leq \alpha$; in other words,
$$\FD_\alpha \circ \FD_\alpha \subseteq \FD_{\Phi(\alpha)} \qquad (\alpha\in \CI).$$
Then $\fd$ is called a \emph{coarse $\CI$-metric} on $X$, and $(X,\CI, \fd)$ is a called a \emph{coarse metric space}.  

\smnoind
(b) A coarse $\CI$-metric $\fd$ is said to be \emph{saturated} if for each $\alpha, \beta\in \CI$, 
\begin{itemize}
	\item the inclusion $\FD_\alpha \subseteq \FD_\beta$ implies $\alpha \leq \beta$;
	
	\item for any subset $S\subseteq \FD_\alpha$ with $S^{-1} = S$ and $\Delta_X\subseteq S$, one can find $\gamma\in \CI$ such that $S = \FD_\gamma$. 
\end{itemize}

\smnoind
(c) Suppose that $\CI'$ is another upward directed set with a zero and $\fd'$ is a coarse $\CI'$-metric on $X$. 
We say that $\fd$ is \emph{coarsely dominated by} $\fd'$ (and denote this by $\fd \preceq \fd'$) if there is \textcolor{red}{an increasing} map $\Gamma: \CI'\to \CI$ such that $\fd(x,y) \leq \Gamma_\infty(\fd'(x,y))$ for every $x,y\in X$ (see \eqref{eqt:def-Lambda-infty}). 
If $\fd \preceq \fd'$ and $\fd'\preceq \fd$, then we say that $\fd$ is \emph{coarsely equivalent to} $\fd'$ and denote this by $\fd \sim \fd'$. 
\end{defn}

\medskip

\begin{rem}\label{rem:coarse-met}
Let $\fd$ be a coarse $\CI$-metric.

\smnoind
(a) If $\Phi$ is a map satisfying the requirement in Definition \ref{defn:coarse-met}(a), then it is not hard to see that $\fd(x,y)\leq \Phi(\fd(x,y))$ ($x,y\in X$). 
Moreover, as $\CI$ is upward directed, one can always find a map $\Phi$ with $\alpha\leq \Phi(\alpha)$ ($\alpha\in \CI$) that satisfies the requirement of Definition \ref{defn:coarse-met}(a). 

\smnoind
(b) Suppose, in addition, that $\CI$ is meet-complete. 
Then clearly, $\bigcap_{\alpha\in \CS} \FD_\alpha \subseteq \FD_{\inf \CS}$ for $\CS\subseteq \CI$. 
We set 
$$\CI(\alpha):= \{\beta\in \CI: \FD_\alpha \circ \FD_\alpha \subseteq \FD_\beta \} \qquad (\alpha\in \CI),$$
and define $\hat \Phi(\alpha):= \inf \CI(\alpha)$. 
It is obvious that $\hat \Phi$ is increasing. 
If $x,y,z\in X$ satisfying $\fd(x,y)\leq \alpha$ and $\fd(y,z)\leq \alpha$, then $\fd(x,z)\leq \beta$ for any $\beta\in \CI(\alpha)$, and hence $\fd(x,z)\leq \hat \Phi(\alpha)$. 
Thus, in the case when $\CI$ is meet-complete, we can always find an increasing map $\Phi$ satisfying the requirement in Definition \ref{defn:coarse-met}(a). 

\smnoind
\textcolor{red}{(c) It makes no harm to assume that the directed set $\CI$ where a coarse metric $\fd$ takes values is meet complete. 
In fact, for every $\alpha\in \CI$, we set $\tilde \alpha:=\{\beta\in \CI: \beta\leq \alpha\}\in \CP_0(\CI)$.
Denote  
$$\tilde \CI:= \{A\in \CP_0(\CI): A\subseteq \tilde \alpha, \text{ for some }\alpha \in \CI\},$$
and define $j:\CI \to \tilde \CI$ to be the map sending $\alpha$ to $\tilde \alpha$. 
Then $\tilde \CI$ is a meet complete upward directed set with zero (namely, $\tilde 0$) and $j$ is an order preserving injection.
We set $\tilde \Phi: \tilde \CI \to \tilde \CI$ such that 
$$\ti \Phi(A) := \bigcap \{j(\Phi(\alpha)): \alpha \in \CI; A \subseteq \ti \alpha\} \qquad (A\in \ti \CI). $$ 
If we put $\tilde \fd := j\circ \fd$, then $\tilde \fd$ is a coarse $\tilde \CI$-metric.
For every $\alpha \in \CI$, one has $\FD_\alpha = \FD_{\tilde \alpha}$. 
On the other hand, for each $A\in \tilde \CI$, there exists $\alpha\in \CI$ with $A \subseteq \tilde \alpha$, and hence $\FD_A \subseteq \FD_{\tilde \alpha} = \FD_\alpha$. 
Thus, the coarse structures generated by $\{\FD_A: A\in \ti \CI \}$ and by $\{\FD_\alpha:\alpha\in \CI\}$ are the same.} 
\end{rem}

\medskip

Let us begin will the following easy fact.

\medskip

\begin{lem}\label{lem:meet-clo-tot-ord}
Let $\CT$ be a subcollection of $\CP_0(X\times X)$. 
We set $\CT_\mathrm{s} := \{(A\cap A^{-1})\cup \Delta_X: A\in \CT \}$ as well as  
$$\overline{\CT} := \Big\{\bigcap \CS: \emptyset \neq \CS\subseteq \CT_\mathrm{s} \Big\}.$$
The collection $\overline{\CT}$ is closed under arbitrary intersections (i.e. the intersection of any subset of $\overline{\CT}$ belongs to $\overline{\CT}$). 
Moreover, if $\CT$ is totally ordered, then so is $\overline{\CT}$. 
\end{lem}
\begin{proof}
The first claim is obvious. 
Suppose now that $\CT$ is totally ordered.
Then $\CT_\mathrm{s}$ is also totally ordered. 
Consider $\CC, \D\subseteq \CT_\mathrm{s}$. 
Then either there exists $C_0\in \CC$ such that $C_0\subseteq D$ for every $D\in \D$, or for each $C\in \CC$, one can find $D\in \D$ with $D\subseteq C$. 
In the first case, we have $\bigcap \CC \subseteq C_0 \subseteq \bigcap \D$. 
In the second case, we know that $\bigcap \D \subseteq \bigcap \CC$. 
\end{proof}

\medskip

\begin{thm}\label{thm:coarse metric}
Let $X$ be a set. 

\smnoind
(a) Suppose that $\CI$ is a \textcolor{red}{meet complete} upward directed set with a zero and $\fd$ is a coarse $\CI$-metric on $X$. 
Then 
$\CB_\fd \cup \{\FD_{0_\CI}\}$ 
(see Defintion \ref{defn:gen-met}(a)) is a base for a coarse structure $\CE_\fd$ on $X$. 
Moreover, if $\CI'$ is another upward directed set with a zero and $\fd'$ is a $\CI'$-metric on $X$, then $\fd$ is coarsely equivalent to $\fd'$ if and only if $\CE_\fd = \CE_{\fd'}$. 

\smnoind
(b) Let $\CE$ be a coarse structure on $X$. 
There is a unique upward directed set $\CI^\CE$ with a zero such that one can find a (necessarily unique) saturated coarse $\CI^\CE$-metric $\fd^\CE$ with $\CE = \CE_{\fd^\CE}$. 
In this case, $\CI^\CE_\infty$ is a complete lattice. 

\smnoind
(c) Let $\CB$ be a base for a coarse structure $\CE$ on $X$, and $\overline{\CB}$ be as in Lemma \ref{lem:meet-clo-tot-ord}. 
Then $\overline{\CB}$ is a meet-complete upward directed set, and there is a coarse $\overline{\CB}$-metric $\fd^\CB$ such that $\CE = \CE_{\fd^\CB}$. 
\end{thm}
\begin{proof}
(a) For any $\alpha, \beta\in \CI$, if $\gamma\in \CI$ satisfying $\alpha\leq \gamma$ and $\beta\leq \gamma$, then $\FD_\alpha \circ \FD_\beta \subseteq \FD_{\Phi(\gamma)}$. 
This gives the first statement. 
For the second statement, we note $\fd \preceq \fd'$ if and only if $\CE_{\fd'}\subseteq \CE_\fd$ \textcolor{red}{(for the backward implication, we define $\Gamma(\alpha'):=\inf\{\alpha\in \CI: \FD_{\alpha'} \subseteq \FD_\alpha\}$ for every $\alpha'\in \CI'$)}. 

\smnoind
(b) As $\CE$ is closed under the formation of finite unions and subsets, we see that 
\begin{equation}\label{eqt:def-I-E}
\CI^\CE := \{E\in \CE: E^{-1} = E; \Delta_X\subseteq E \}
\end{equation}
is a meet-complete lattice (under inclusion), and it contains $\Delta_X$ as its smallest element (i.e. zero). 
Moreover, since every subset of $\CI^\CE$ that has a upper bound in $\CI^\CE$ has a least upper bound in $\CI^\CE$, we know that $\CI^\CE_\infty$ is a complete lattice.
Set 
$$\fd^\CE(x,y) := \begin{cases}
\Delta_X \cup \{(x,y), (y,x) \}  \quad &\text{when }(x,y)\in E \text{ for some } E\in \CE\\
\infty  &\text{otherwise.}
\end{cases}
$$
Obviously, $\fd^\CE: X\times X\to \CI^\CE_\infty$ is a semi-$\CI^\CE$-metric. 
If we put $\Phi^\CE(E):= E\circ E\in \CI^\CE$ for each $E\in \CI^\CE$, then $\Phi^\CE$ will satisfy the requirement in Definition \ref{defn:coarse-met}(a), and $\fd^\CE$ is a coarse $\CI^\CE$-metric. 
Furthermore, as 
\begin{equation}\label{eqt:E-diag}
E = \{(x,y)\in X\times X: \fd^\CE(x,y)\leq E \} \qquad (E\in \CI^\CE),
\end{equation}
we see that $\fd^\CE$ is saturated and that $\CE = \CE_{\fd^\CE}$. 

Suppose now that $\CI$ is another upward directed set with a zero and $\fd$ is a saturated coarse $\CI$-metric on $X$ with $\CE = \CE_\fd$. 
Then the saturation assumption of $\fd$ implies that 
$\alpha \mapsto \FD_\alpha$ is an order isomorphism from $\CI$ onto $\CI^\CE$.
Furthermore, it follows from the definitions and \eqref{eqt:E-diag} that for any $u,v\in X$ and $\beta\in \CI$, one has 
\begin{equation}\label{eqt:d=d-CE}
\fd(u,v)\leq \beta\quad \text{if and only if} \quad \fd^\CE(u,v)\leq \FD_\beta.
\end{equation}
For every $x,y\in X$, it is not hard to verify, through \eqref{eqt:d=d-CE}, that $\fd^\CE(x,y) = \FD_{\fd(x,y)}$. 
In other words, $\fd$ is the same as $\fd^\CE$ under the order isomorphism $\alpha\mapsto \FD_\alpha$. 

\smnoind
(c) The meet-completeness of $\overline{\CB}$ follows from Lemma \ref{lem:meet-clo-tot-ord}. 
Suppose that $\CS,\CT\in \CP_0(\CB_\mathrm{s})$, where $\CB_\mathrm{s}$ is as in Lemma \ref{lem:meet-clo-tot-ord}. 
Take any $S\in \CS$ and $T\in \CT$. 
As $\CB_\mathrm{s}$ is a base for $\CE$, there exists $B\in \CB_\mathrm{s}$ with $S\cup T\subseteq B$. 
Then $B\in \overline{\CB}$ and $(\bigcap \CS)\cup (\bigcap \CT)\subseteq B$. 
This shows that $\overline{\CB}$ is upward directed. 
Let us define 
\begin{equation}\label{eqt:defn-fd-CB}
\fd^\CB(x,y) := \bigcap \{D\in \overline{\CB}: (x,y)\in D \} \qquad (x,y\in X);
\end{equation}
here we use the convention that $\bigcap \emptyset = \infty$. 
We also set 
$$\Phi^\CB(B):= \bigcap \{D\in \overline{\CB}: B\circ B\in D \} \qquad (B\in \overline{\CB}).$$
Clearly, $\Phi^\CB$ satisfies the requirement in Definition \ref{defn:coarse-met}(a), and $\fd^\CB$ is a coarse $\overline{\CB}$-metric. 
Moreover, as $B = \{(x,y)\in X\times X: \fd^\CB(x,y)\leq B \}$ for any $B\in \overline{\CB}$, we know that $\CE = \CE_{\fd^\CB}$. 
\end{proof}

\medskip

Obviously, in the case when $\CI\neq \{0_\CI\}$ and $\fd$ is a coarse-$\CI$-metric, then $\CB_\fd$ is a base for $\CE_\fd$. 
On the other hand, although $\CI^\CE_\infty$ in part (b) above is a complete lattice, it does not mean that $\CI^\CE\cup \{X\}$ is closed under arbitrary unions. 
In fact, if $\CS\subseteq \CI^\CE$ such that $\bigcup \CS\notin \CI^\CE$, then the least upper bound of $\CS$ in $\CI^\CE_\infty$ is $\infty$.

\medskip

\begin{rem}
The maps $\Phi^\CE$ and $\Phi^\CB$ in the proof of parts (b) and (c) of Theorem \ref{thm:coarse metric} are increasing and satisfy the requirement in Definition \ref{defn:coarse-met}(a) for $\fd^\CE$ and $\fd^\CB$, respectively. 
Moreover, one has $E\subseteq \Phi^\CE(E)$ (respectively, $B\subseteq \Phi^\CB(B)$)  for every $E\in \CI^\CE$ (respectively, $B\in \overline{\CB}$). 
\end{rem}

\medskip

The following example tells us that if $\fd$ is a coarse $\CI$-metric and $\CJ$ is a upward directed set containing $\CI$, then the coarse structure induced by $\fd$ when $\fd$ is considered as a coarse $\CJ$-metric may not be the same as the one when $\fd$ is considered as a coarse $\CI$-metric.

\medskip

\begin{eg}\label{eg:diff-dir-set}
Let $(\BR,\fd_1)$ be the Euclidean metric space. 
Clearly, $\fd_1$ is a $\RP$-metric and the coarse structure generated by this $\RP$-metric is the usual one. 

However, if we set $\CJ = \RP_\infty$, then $\fd_1$ is also a $\CJ$-metric, but the coarse structure generated by this $\CJ$-metric is the ``trivial one'', because $\BR = \FD_\infty$ is a controlled set. 
\end{eg}

\medskip

One can express many concepts concerning coarse structures in terms of metric, which seem easier to understand and handle. 
Let us list some of them in the following. 

\medskip

\begin{prop}\label{prop:coarse-prop-met}
Let $(X,\CE)$ and $(Y,\CF)$ be two coarse spaces. 
Suppose that $\fd_X$ (respectively, $\fd_Y$) is a coarse $\CI$-metric on $X$ (respectively, coarse $\CJ$-metric on $Y$) that induces the underlying coarse structure. 
Let $f,g:X\to Y$ be two maps. 

\smnoind
(a) $(X, \CE)$ is coarsely connected if and only if the largest element $\infty \in \CI_\infty$ does not belong to  
$\fd_X(X\times X)$.

\smnoind
(b) $B\subseteq X$ is bounded if and only if one can find $(x,\alpha)\in X\times \CI$ with $B\subseteq \FD(x,\alpha)$. 

\smnoind
(c) \textcolor{red}{If $\CJ$ is meet complete, then} $f$ is bornologous if and only if $\fd_Y \circ (f\times f) \preceq \fd_X$.

\smnoind
(d) $f$ is proper if and only if there is a map $\Upsilon: Y\times \CJ \to X\times \CI$ with $f^{-1}(\FD(y,\beta))\subseteq \FD(\Upsilon(y,\beta))$, for each $(y,\beta)\in Y\times \CJ$.

\smnoind
(e) \textcolor{red}{If $\CI$ is meet complete, then} $f$ is effectively proper if and only if $\fd_X\preceq \fd_Y\circ (f\times f)$.

\smnoind
(f) $f$ and $g$ are close if and only if there exists $\beta\in \CJ$ with $\fd_Y(f(x),g(x))\leq \beta$, for any $x\in X$. 
\end{prop}
\begin{proof}
(a) Recall that $(X,\CE)$ is coarsely connected if and only if $\{(x,y)\}\in \CE$, for every $x,y\in X$. 
Clearly, this is equivalent to $\fd_X(x,y)\in \CI$, for every $x,y\in X$. 

\smnoind
(b) Recall that $B$ is bounded if and only if $B\subseteq E[x]$ for some $E\in \CE$ and $x\in X$. 
The equivalence in the statement is more or less trivial. 

\smnoind
(c) Recall that $f$ is bornologous if and only if $(f\times f)(\CE)\subseteq \CF$; equivalently, for every $\alpha\in \CI$, one can find $\textcolor{red}{\beta}\in \CJ$ such that $(f\times f)(\FD_\alpha)\subseteq \FD_{\textcolor{red}{\beta}}$.
\textcolor{red}{In this case, the map $\Gamma$ defined by  $\Gamma(\alpha):=\inf \{\beta\in \CJ: (f\times f)(\FD_\alpha)\subseteq \FD_{\beta} \}$ $(\alpha\in \CI$) is increasing.}
Now, for \textcolor{red}{an increasing} map $\Gamma:\CI\to \CJ$, the condition $(f\times f)(\FD_\alpha)\subseteq \FD_{\Gamma(\alpha)}$ is equivalent to the requirement as in Definition \ref{defn:coarse-met}(c) for $\fd_Y \circ (f\times f) \preceq \fd_X$.  

\smnoind
(d) Recall that $f$ is proper if and only if $f^{-1}(B)$ is bounded for any bounded set $B\subseteq Y$. 
Thus, this part follows directly from part (b). 

\smnoind
(e) Recall that $f$ is effectively proper if and only if $(f\times f)^{-1}(\CF)\subseteq \CE$. 
This is the same as saying that for $\beta\in \CJ$, one find $\textcolor{red}{\alpha}\in \CI$ such that $(f\times f)^{-1}(\FD_\beta)\subseteq \FD_{\textcolor{red}{\alpha}}$, or equivalently, 
\begin{quotation}
$\fd_X(x,y)\leq \textcolor{red}{\alpha}$ whenever $\fd_Y(f(x), f(y))\leq \beta$. 
\end{quotation}
\textcolor{red}{In this case, the map $\Gamma$ defined by  $\Gamma(\beta):=\inf \{\alpha\in \CI: (f\times f)^{-1}(\FD_\beta)\subseteq \FD_{\alpha}\}$ $(\beta\in \CJ$) is increasing.}
Now, \textcolor{red}{an increasing} map $\Gamma:\CJ \to \CI$ satisfies the above displayed statement \textcolor{red}{for $\beta = \Gamma(\alpha)$} if and only if it satisfies the requirement in Definition \ref{defn:coarse-met}(c) for \textcolor{red}{$\fd_X \preceq\fd_Y \circ (f\times f)$}.  

\smnoind
(f) Recall that $f$ and $g$ are close if and only if $\{(f(x), g(x)): x\in X \}\in \CF$, which is obviously the same as the requirement in the statement of part (f). 
\end{proof}

\medskip

Notice that in Example \ref{eg:diff-dir-set}, although the range of $\fd_1$ take the value ``$\infty$'', it is the largest element in $\CJ=\BR^+_\infty$, but not the largest element of $\CJ_\infty$. 
Therefore, the resulting coarse structure is connected. 

\medskip

On the other hand, we can ``reverse-engineer'' some concepts in coarse structure back to metric space terms. 
The following is such an example. 
Let us recall from \cite[Definition 3.9]{Roe} that a coarse space $(X,\CE)$ is said to have \emph{bounded geometry} if one can find $E\in \CE$ containing $\Delta_X$ such that $E^{-1} = E$ and
\begin{equation}\label{eqt:def-bdd-geom}
{\sup}_{x\in X}\max \big\{\mathrm{cap}_E((F\circ E)[x]), \mathrm{cap}_E((F^{-1}\circ E)[x])\big\} < \infty \qquad (F\in \CE),
\end{equation}
where $\mathrm{cap}_E(S) := \sup \{ m\in \BN: \text{there exist } y_1,...,y_m\in S \text{ with }(y_i,y_j)\notin E \text{ when } i\neq j\}$. 

\medskip

Clearly,  $\mathrm{cap}_{E'}(S) \leq \mathrm{cap}_E(S)$ if $E\subseteq E'$ and 
$\mathrm{cap}_{E}(S) \leq \mathrm{cap}_E(S')$ if $S\subseteq S'$. 
Moreover, one has $F\subseteq F\circ E\in \CE$. 
Consequently, one may replace \eqref{eqt:def-bdd-geom} with 
${\sup}_{x\in X}\mathrm{cap}_E(F[x]) < \infty$, for every $F\in \CE$ with $F^{-1} = F$. 
Thus, in the case when $\CE$ is defined by a coarse $\CI$-metric, $(X,\CE)$ has boudned geometry if and only if one can find $\alpha_1\in \CI$ satisfying 
\begin{equation}\label{eqt:equiv-bdd-geom}
{\sup}_{x\in X} \mathrm{cap}_{\FD_{\alpha_1}}(\FD(x,\alpha)) < \infty \qquad (\alpha\in \CI).
\end{equation}

\medskip

\begin{prop}\label{prop:bdd-geom}
Let $(X,\CE)$ be a coarse space and $\fd_X$ be a coarse $\CI$-metric defining $\CE$. 
The following statements are equivalent. 

	\begin{enumerate}[B1)]
		\item $(X, \CE)$ has bounded geometry. 
		
		\item There is $\alpha_1\in \CI$ satisfying: for any $\alpha\in \CI$, there exists $n_1\in \BN$ such that each ball of radius $\alpha$ contains at most $n_1$ points 
		with their pairwise $\fd_X$-distances not dominated by $\alpha_1$ (i.e. $\fd_X(x,y)\nleq \alpha_1$). 
		
		\item There is $\alpha_2\in \CI$ satisfying: for any $\alpha\in \CI$, there exists $n_2\in \BN$ such that for each $x\in X$, the ball $\FD(x,\alpha)$ contains at most $n_2$ disjoint relative balls of radius $\alpha_2$ (here, relative balls of radius $\alpha_2$ are subsets of the form $\FD(x,\alpha)\cap \FD(y,\alpha_2)$ for some $y\in \FD(x,\alpha)$). 
		
		\item There is $\alpha_3\in \CI$ satisfying: for any $\alpha\in \CI$, there exists $n_3\in \BN$ such that each ball of radius $\alpha$ is contained in the union of $n_3$ balls of radius $\alpha_3$. 
	\end{enumerate}
\end{prop}
\begin{proof}
$(B1)\Leftrightarrow (B2)$ This equivalence is simply a matter of reformulating \eqref{eqt:equiv-bdd-geom}. 

\smnoind
$(B2)\Rightarrow (B3)$ 
If $\FD(x,\alpha)$ contains at most $n_1$ points with their mutual $\fd_X$-distances not dominated by $\alpha_1$, then clearly, it cannot contain more than $n_1$ disjoint relative balls of radius $\alpha_1$. 

\smnoind
$(B3)\Rightarrow (B2)$ 
Suppose that $\FD(x,\alpha)$ contains at most $n_2$ disjoint relative balls of radius $\alpha_2$. 
Let $\Phi$ be as in Definition \ref{defn:coarse-met}(a) and $\alpha_1 := \Phi(\alpha_2)$. 
Then $\FD(x,\alpha)$ cannot contain more than  $n_2$ points with their mutual $\fd_X$-distance not dominated by $\alpha_1$. 

\smnoind
$(B1)\Leftrightarrow (B4)$
Let us recall from \cite[Definition 3.1(a)]{Roe} the following definition: 
$$\mathrm{ent}_E(S):= \inf \{n\in \BN: \text{ there exist }x_1,...,x_n\in X \text{ with } S \subseteq E[x_1]\cup \cdots \cup E[x_n] \}$$
(where $\inf \emptyset := \infty$). 
It was shown in \cite[Proposition 3.2(d)]{Roe} that 
$$\mathrm{cap}_{E\circ E}(S) \leq \mathrm{ent}_E(S) \leq \mathrm{cap}_E(S).$$
Therefore, one can replace ``$\mathrm{cap}$'' by ``$\mathrm{ent}$'' in the definition of bounded geometry; namely, 
$${\sup}_{x\in X}\mathrm{ent}_E(F[x]) < \infty, \quad \text{ for every }F\in \CE \text{ with } F^{-1} = F.$$ 
In other words, $(X, \CE)$ has bounded geometry if and only if 
there exists $\alpha_3\in \CI$ such that 
\begin{equation*}
{\sup}_{x\in X} \mathrm{ent}_{\FD_{\alpha_3}}(\FD(x,\alpha)) < \infty \qquad (\alpha\in \CI).
\end{equation*}
This statement is clearly equivalent to Statement (B4). 
\end{proof}

\medskip

Using the equivalence of Statements (B1) and (B2), one obtains an easy way to see that every bounded geometry space is coarse equivalent to a uniformly discrete space. 
In fact, consider $\CB = \{\FD(x_i, \alpha_1): i\in \KI \}$ to be a maximal collection of disjoint balls of radius $\alpha_1$.  
Since $\bigcup_{i\in \KI} \FD(x_i, \Phi(\alpha_1)) = X$, we know that $X$ is coarse equivalent to its subspace  $\{x_i:i\in \KI \}$ (as in Remark \ref{rem:coarse-met}(a), we may assume that $\alpha_1\leq \Phi(\alpha_1)$), and the later is uniformly discrete. 

\medskip

We end this section with a discussion of the coarse structure induced on the collection $\CP_0(X)$ of non-empty subsets of a coarse space $X$. 

\medskip

\begin{defn}\label{defn:Hausd-coar-str}
Let $(X,\CE)$ be a coarse space. 
For any $E\in \CI^\CE$ (see \eqref{eqt:def-I-E}), we set 
\begin{equation}\label{eqt:def-check-E}
\check E:= \big\{(R,S)\in \CP_0(X)\times \CP_0(X): R \subseteq E[S] \text{ and } S \subseteq E[R] \big\}.
\end{equation}
The coarse structure  $\check \CE$ on $\CP_0(X)$ generated by $\{\check E: E\in \CI^\CE \}$ is called the \emph{Hausdorff coarse structure associated with $\CE$}. 
\end{defn}

\medskip

Notice that if $\CB$ is a base for $\CE$, then $\{\check B: B\in \CB_\mathrm{s} \}$ (see Lemma \ref{lem:meet-clo-tot-ord}) is a base for $\check \CE$. 

\medskip 

By Theorem \ref{thm:coarse metric}(b),  both $\CE$ and $\check \CE$ are defined by coarse metrics.
It is natural to ask the relation between metrics defining $\CE$ and those defining $\check \CE$. 
In the case when $\CI$ is meet-complete, one natural guess is the following \emph{Hausdorff semi-metric} associated with a coarse $\CI$-metric $\fd$: 
\begin{equation}\label{eqt:defn-check-d}
\check \fd(R,S):= \inf \Big\{\alpha\in \CI: R \subseteq {\bigcup}_{s\in S} \FD(s, \alpha) \text{ and } S \subseteq {\bigcup}_{r\in R} \FD(r, \alpha) \Big\} \qquad (R,S\in \CP_0(X));
\end{equation}
again, we use the convention that $\inf \emptyset = \infty$. 
Two natural questions are: 
\begin{enumerate}
	\item is $\check \fd$ actually a coarse $\CI$-metric? 
	
	\item does the coarse structure defined by $\check \fd$ coincide with $\check \CE_\fd$? 
\end{enumerate}
We doubt if these two questions have positive answers in gerenal. 
However, we will consider a situation when they do.

\medskip

\begin{prop}\label{prop:Hausd-coar-met}
Let $\CI$ be a meet-complete totally ordered set, and $\fd$ is a coarse $\CI$-metric on a set $X$. 

\smnoind
(a) $\check \fd$ is a coarse $\CI$-metric on $\CP_0(X)$. 

\smnoind
(b) The coarse structure induced by $\check \fd$ is precisely $\check \CE_\fd$. 
\end{prop}
\begin{proof}
(a) It is clear that 
$\check \fd(R,R) = 0_\CI$ and $\check \fd(R,S) = \check \fd(S,R)$ for $R,S\in \CP_0(X)$.  
By Remark \ref{rem:coarse-met}(b), there is an increasing map $\Phi:\CI\to \CI$ satisfying the requirement in Definition \ref{defn:coarse-met}(a). 
Consider $\alpha\in \CI$. 
When $\alpha$ is the largest element of $\CI$ (if it exists), then we set $\ti \Phi(\alpha) := \alpha$. 
When $\alpha$ is not the largest element of $\CI$, we fix an element $\beta(\alpha)\in \CI$ with $\alpha\lneq \beta(\alpha)$ and set $\ti \Phi(\alpha) := \Phi(\beta(\alpha))$. 
Let $R,S,T\in \CP_0(X)$. 
Consider 
$$\CI(R,S):= \Big\{\delta\in \CI: R \subseteq {\bigcup}_{s\in S} \FD(s, \delta) \text{ and } S \subseteq {\bigcup}_{r\in R} \FD(r, \delta) \Big\}.$$
Assume that $\check \fd(R,S)\leq \alpha$ and $\check \fd(S,T)\leq \alpha$. 
Since $\alpha \in \CI$, we know that $\CI(R,S)\neq \emptyset$ and $\CI(S,T)\neq \emptyset$. 
If $\alpha$ is the largest element of $\CI$, then obviously, $\check \fd(R,T)\leq \alpha = \ti \Phi(\alpha)$. 
Otherwise, since $\check \fd(R,S)\lneq \beta(\alpha)$, there exists $\gamma'\in \CI(R,S)$ with $\gamma'\leq \beta(\alpha)$ (because $\CI$ is totally ordered). 
Similarly, one can find $\gamma''\in \CI(S,T)$ with $\gamma''\leq \beta(\alpha)$. 
Hence, if $\gamma:= \max\{\gamma',\gamma'' \}$, then $R \subseteq {\bigcup}_{t\in T} \FD(t, \Phi(\gamma))$ and $T \subseteq {\bigcup}_{r\in R} \FD(r, \Phi(\gamma))$. 
This implies that $\check \fd(R,T)\leq \Phi(\gamma) \leq \ti \Phi(\alpha)$, because $\Phi$ is increasing. 

\smnoind
(b) As said in the above, $\big\{\check \FD_{\alpha}: \alpha\in \CI \big\}$ is a base for $\check \CE_\fd$, where $\check \FD_{\alpha}$ is defined as in \eqref{eqt:def-check-E}. 
Fix $\alpha_0\in \CI$.
When $\alpha_0$ is the largest element in $\CI$ (if it exists), one has 
\begin{align*}
\check \FD_{\alpha_0} & = \Big\{(R,S)\in \CP_0(X)\times \CP_0(X): R \subseteq {\bigcup}_{s\in S} \FD(s, \beta) \text{ and } S \subseteq {\bigcup}_{r\in R} \FD(r, \beta), \text{ for some }\beta\in \CI \Big\}\\
& = \{(R,S) \in \CP_0(X)\times \CP_0(X):  \check \fd(R,S)\leq \alpha_0 \}. 
\end{align*}
Assume that $\alpha_0$ is not the largest element in $\CI$. 
Choose any $\beta\in \CI$ with $\alpha_0\lneq \beta$. 
If $R,S\in \CP_0(X)$ satisfying $\check \fd(R,S)\leq \alpha_0$, then,  as in the argument of part (a), one can find $\gamma\in \CI(R,S)$ with $\gamma\leq \beta$. 
From this, we know that $(R,S)\in \check \FD_\beta$. 
Conversely, if $(R,S)\in \check \FD_{\alpha_0}$, then it is clear that $\check \fd(R,S)\leq \alpha_0$. 
\end{proof}

\medskip

We end this section with the following result concerning the case when $\CE$ has a totally ordered base. 
Note that part (a)  of this result comes from Lemma \ref{lem:meet-clo-tot-ord}, as well as  parts (a) and (c) of Theorem \ref{thm:coarse metric}. 
Part (b) is a corollary of Proposition \ref{prop:Hausd-coar-met} and Theorem \ref{thm:coarse metric}(c).

\medskip

\begin{cor}\label{cor:Haus-met}
Let $\CE$ be a coarse structure on a set $X$. 
	
\smnoind
(a) $\CE$ has a totally ordered base $\CB$ if and only if there is a meet-complete totally ordered set $\CI$ and a coarse $\CI$-metric $\fd$ with $\CE = \CE_\fd$. 
In this case, $(\CI, \fd)$ can be chosen to be $(\overline{\CB}, \fd^\CB)$ (see Theorem \ref{thm:coarse metric}(c)). 

\smnoind
(b) Suppose that $\CE$ admits a totally ordered base $\CB$. 
Then $\check \fd^\CB$ is a coarse $\overline{\CB}$-metric on $\CP_0(X)$, and the coarse structure induced by $\check \fd^\CB$ is precisely the Hausdorff coarse structure associated with $\CE$. 
\end{cor}

\medskip

\section{Uniform $\CI$-metric}

\medskip

In this section, we consider the uniform structure that comes from some form of metric. 
Let us recall that a \emph{uniform structure} on a set $X$ is a subcollection $\CU\subseteq \CP_0(X\times X)$ such that for every $U,V\in \CU$ and $S\in \CP_0(X\times X)$ with $U\subseteq S$, one has $\Delta_X\subseteq U$, $U\cap V,U^{-1},S\in \CU$ and there exists $W\in \CU$ with $W\circ W\subseteq U$.  
A subcollection $\CB\subseteq \CU$ is called a \emph{base} for $\CU$ if for every $U\in \CU$, there exist $V\in \CB$ with $V\subseteq U$. 

\medskip

\begin{defn}\label{defn:unif-met}
(a) If $\CI$ is a partially ordered set with a zero such that $\CI\setminus \{0_\CI\}$ is non-empty and is downward  directed, then we say that $\CI$ is a \emph{D-index set}. 

\smnoind
(b) Suppose that $\CI$ is a D-index set and $\fd$ is a semi-$\CI$-metric. 
If there is a function $\Psi: \CI\setminus \{0_\CI\}\to \CI\setminus \{0_\CI\}$ such that for any $\beta\in \CI\setminus \{0_\CI\}$, one has $\fd(x,z)\leq \beta$ whenever $x,y,z\in X$ satisfying $\fd(x,y)\leq \Psi(\beta)$ and $\fd(y,z)\leq \Psi(\beta)$; i.e,
$$\FD_{\Psi(\beta)}\circ \FD_{\Psi(\beta)}\subseteq \FD_{\beta} ,$$
then $\fd$ is called a \emph{pseudo uniform-$\CI$-metric}. 

\smnoind
(c) A \emph{uniform $\CI$-metric} is a pseudo uniform-$\CI$-metric $\fd$ satisfying 
$\bigcap_{\alpha\in \CI\setminus \{0_\CI\}} \FD_\alpha = \Delta_X$. 
\end{defn}

\medskip

\begin{rem}\label{rem:no-atom}
If $\CI\setminus  \{0_\CI\}$ does not have a smallest element, then $\inf \CI\setminus \{0_\CI\}$ exists and equals $0_\CI$. 
In this case, a pseudo uniform $\CI$-metric is a uniform $\CI$-metric if and only if the relation $\fd(x,y) = 0_\CI$ implies $x = y$. 
\end{rem}

\medskip

The following result is more or less trivial. 

\medskip

\begin{prop}\label{prop:unif-metric}
Let $X$ be a set, and  $\CI$ be a D-index set. 
Suppose that $\fd$ is a pseudo uniform $\CI$-metric on $X$.
Then $\CB_\fd$ (see Definition \ref{defn:gen-met}(a)) is a base for a uniform structure $\CU_\fd$ on $X$. 
If, in addition, $\fd$ is a uniform $\CI$-metric, then the topology induced by $\CU_\fd$ is Hausdorff. 
\end{prop}

\medskip

We say that a uniform structure $\CU$ is \emph{trivial} if it is a principal filter (i.e. there exists $U_0\in \CU$ with $\CU = \{U\subseteq X\times X: U_0\subseteq U \}$);
otherwise, $\CU$ is said to be \emph{non-trivial}.  
It is clear that $\CU$ is non-trivial if and only if 
$$0_\CU := \bigcap \CU \notin \CU.$$
If $\CU$ is a trivial uniform structure and we define $\fd:X\times X\to \BR^+$ by 
$$\fd(x,y) := \begin{cases}
0 & \text{ when } (x,y)\in 0_\CU\\
1 & \text{ otherwise,}
\end{cases}$$
then $\CU = \CU_\fd$ (because $\bigcap \CU = \FD_{1/2}$). 

\medskip

\begin{lem}\label{lem:base-closed-intersect}
Let $\CU$ be a non-trivial uniform structure. 
Suppose that $\CA$ is a base for $\CU$ satisfying:
\begin{equation*}
\bigcap \{A\in \CA: (x,y)\in A \} \in \CU, \quad \text{for every }(x,y)\in X\times X\setminus 0_\CU.
\end{equation*}
If $\CB := \overline{\CA}\setminus \{0_\CU\}$ (see Lemma \ref{lem:meet-clo-tot-ord}), then $\CB$ is a base for $\CU$. 
\end{lem}
\begin{proof}
Clearly, the collection $\CA_\mathrm{s}$ as in Lemma \ref{lem:meet-clo-tot-ord} is a base for $\CU$ and we have $\CA_\mathrm{s}\subseteq \CB$ (as $0_\CU\notin \CA_\mathrm{s}$ by the assumption of non-triviality).
Therefore, it remains to show that $\CB\subseteq \CU$. 

Pick any $(x,y)\in X\times X\setminus 0_\CU$. 
As $(y,x)\notin 0_\CU$, one has 
$$\{B\in \CA_\mathrm{s}: (x,y)\in B \} = \{A\cap A^{-1}: A \in \CA; (x,y)\in A \}\cap \{A\cap A^{-1}: A \in \CA; (y,x)\in A \}.$$
This shows that 
$\bigcap \{B\in \CA_\mathrm{s}: (x,y)\in B \} \in \CU$.

Suppose now that $D\in \CB$. 
There exists $\CC\subseteq \CA_\mathrm{s}$ with $D = \bigcap \CC$. 
As $0_\CU \subsetneq D$, there exists $(x,y)\in \bigcap \CC \setminus  0_\CU$. 
It then follows from $\CC \subseteq \{B\in \CA_\mathrm{s}: (x,y)\in B \}$ that 
$$\bigcap \{B\in \CA_\mathrm{s}: (x,y)\in B \}\subseteq D$$
and the above gives $D \in \CU$. 
\end{proof}

\medskip

\begin{thm}\label{thm:unif-str=>met}
Let $\CU$ be a non-trivial uniform structure on a set $X$.  

\smnoind
(a) There exist a partially ordered set $\CI$ with a zero and a map $\fd:X\times X\to \CI_\infty$ such that 
$\CB_\fd$ (see Definition \ref{defn:gen-met}(a))
becomes a base for $\CU$ if and only if $\CU$ admits a base $\CB$ such that $\CB\cup \{0_\CU\}$ is closed under arbitrary intersections. 
In this case, one can choose $\CI$ to be the meet-complete D-index set 
$\CJ^\CB:= \CB_\mathrm{s}\cup \{0_\CU\}$ (see Lemma \ref{lem:meet-clo-tot-ord})
and $\fd$ to be a pseudo uniform $\CI$-metric. 

\smnoind
(b) If $\CU$ admits a base $\CB$ with $\CB\cup \{0_\CU\}$ being closed under arbitrary intersections and the topology induced by $\CU$ is Hausdorff, then one can find a uniform $\CJ^\CB$-metric $\fd$ with $\CU = \CU_{\fd}$. 
\end{thm}
\begin{proof}
(a) Suppose that such a map $\fd$ exists. 
For any $(x,y)\in X\times X \setminus 0_\CU$, as $0_\CU = \bigcap \CB_\fd$, 
we have 
\begin{equation*}\label{eqt:D-subset-cap}
\fd(x,y)\neq 0_\CI \quad \text{and} \quad \FD_{\fd(x,y)}\subseteq \bigcap \{B\in \CB_\fd: (x,y)\in B \}.  
\end{equation*}
Thus, one can apply Lemma \ref{lem:base-closed-intersect} to conclude that $\CB:= \overline{\CB_\fd} \setminus \{0_\CU\}$ is a base for $\CU$. 
Moreover, we know from  Lemma \ref{lem:meet-clo-tot-ord} that $\CB \cup \{0_\CU\} = \overline{\CB_\fd}$ is closed under arbitrary intersections.

Conversely, suppose that such a base $\CB$ exists.  
Then $\CB_\mathrm{s}$ is a base of $\CU$ such that $\CB_\mathrm{s}\cup \{0_\CU\}$ is closed under arbitrary intersections.
As $\CU$ is non-trivial, $\CJ^\CB\setminus \{0_\CU\} =\CB_\mathrm{s}$ is downward directed. 
Define 
\begin{equation*}\label{eqt:def-fd-CB}
\fd_\CB(x,y) := \bigcap \{B\in \CB_\mathrm{s}: (x,y)\in B \}\in \CJ^\CB_\infty \qquad (x,y\in X)
\end{equation*}
(we again use the convention that $\bigcap \emptyset := \infty$). 
It is clear that $\fd_\CB$ is a semi-$\CJ^\CB$-metric (observe that $\bigcap \CB_\mathrm{s} = 0_\CU$). 
Moreover,  when $S\in \CJ^\CB$ and $x,y\in X$, one has $\fd_\CB(x,y) \leq S$ if and only if $(x,y)\in S$. 
Thus,
\begin{equation}\label{eqt:S=D}
\{(x,y)\in X\times X: \fd_\CB(x,y)\leq S \} = S \qquad (S\in \CJ^\CB).
\end{equation}
Consider any $S\in \CB_\mathrm{s}\subseteq \CU$. 
Pick an arbitrary $B\in \CB_\mathrm{s}$ with $B\circ B\subseteq S$. 
If $x,y,z\in X$ satisfying $\fd_\CB(x,y)\leq B$ and $\fd_\CB(y,z)\leq B$, then $(x,z)\in S$, or equivalently, $\fd_\CB(x,y)\leq S$. 
These show that $\fd_\CB$ is a pseudo uniform $\CJ^\CB$-metric.
Moreover, \eqref{eqt:S=D} tells us that $\CU =  \CU_{\fd_\CB}$. 

\smnoind
(b) It follows from \eqref{eqt:S=D} that 
$$0_\CU = {\bigcap}_{B\in \CB_\mathrm{s}} \{(x,y)\in X\times X: \fd_\CB(x,y) \leq B \}.$$ 
From this, we see that $\fd_\CB$ is a uniform $\CJ^\CB$-metric if and only if $0_\CU = \Delta_X$, or equivalently, the topology induced by $\CU$ is Hausdorff. 
\end{proof}

\medskip

It follows from the proof of part (a) above that we have one more equivalent condition of $\CU$ being defined by a pseudo uniform $\CI$-metric: $\CU$ admits a base $\CA$ satisfying the requirement in  Lemma \ref{lem:base-closed-intersect}. 

\medskip

\begin{cor}\label{cor:unif-tot-ord}
Let $\CU$ be a non-trivial uniform structure on a set $X$.  
There exist a totally ordered set $\CI$ with a zero and a pseudo uniform $\CI$-metric $\fd$ on $X$ with $\CU = \CU_{\fd}$ if and only if there is a totally ordered base $\CA$ of $\CU$. 
In the case, we can take $\CI= \overline{\CA}$. 
\end{cor}
\begin{proof}
Suppose that there is such a pseudo uniform $\CI$-metric $\fd$. 
Then, obviously, $\big\{\FD_\alpha: \alpha\in \CI\setminus \{0_\CI\} \big\}$ is a totally ordered base for $\CU$. 
Conversely, suppose that one can find a totally ordered base $\CA$ for $\CU$. 
Let $(x,y)\in X\times X\setminus 0_\CU$. 
As $0_\CU = \bigcap\CA$, there is $A_0\in \CA$ with $(x,y)\notin A_0$. 
If $B\in \CA$ contains $(x,y)$, then $A_0\subseteq B$ (as $\CA$ is totally ordered), and $\CA$ satisfies the hypothesis of Lemma \ref{lem:base-closed-intersect}. 
Hence, $\CB:= \overline{\CA} \setminus \{ 0_\CU\}$ is a base for $\CU$. 
Furthermore, Lemma \ref{lem:meet-clo-tot-ord} tells us that $\overline{\CA}$ is totally ordered and is closed under arbitrary intersections. 
The conclusion now follows from Theorem \ref{thm:unif-str=>met}(a). 
\end{proof}

\medskip

\section{An example}

\medskip

Before we give the example concerning valuation rings, we first consider a connection between coarse metrics and pseudo uniform metrics. 
We recall that a subset $\CS$ of a partially ordered set $\CI$ is \emph{downward cofinal} if for every $\alpha\in \CI$, there exists $\beta\in \CS$ such that $\beta\leq \alpha$. 

\medskip

The following result is more or less obvious. 

\medskip

\begin{prop}\label{prop:coarse>unif}
Let $\CI$ be upward directed set which is also a D-index set.
Suppose that $\fd$ is a coarse $\CI$-metric on $X$. 
If there exists $\Phi:\CI \to \CI$ satisfying the requirement in Definition \ref{defn:coarse-met}(a) such that $\Phi(\CI\setminus \{0_\CI\})$ is a downward  cofinal subset of $\CI\setminus \{0_\CI\}$, then $\fd$ is also a pseudo uniform $\CI$-metric. 
\end{prop}

Consequently, if $\CI$ is a upward directed set which is also a D-index set, and $\fd$ is a coarse $\CI$-metric on $X$ such that for every $\alpha\in \CI$ and $x,y,z\in X$, one has $\fd(x,z)\leq \alpha$ if $\fd(x,y)\leq \alpha$ and $\fd(y,z)\leq \alpha$, then $\fd$ is also a pseudo uniform $\CI$-metric. 
In this case, we called $\fd$ an \emph{pseduo ultra $\CI$-metric}.

\medskip

The following is an example of a pseudo ultra $\CI$-metric with $\CI$ being a totally ordered set.

\medskip

\begin{eg}\label{eg:val-ring}
Suppose that $R$ is a (unital) ring and $\Gamma$ is a totally ordered abelian group. 
Let $\Gamma^0$ be the ordered semi-group obtained by adjoining to $\Gamma$ a new element $\omega$, such that $\omega$ is greater than all elements in $\Gamma$ and that $\alpha + \beta = \omega$ when either $\alpha = \omega$ or $\beta = \omega$. 
As in \cite[Definition VI.3.1]{Bour}, a map $\nu: R\to \Gamma^0$ is called a \emph{valuation} if for any $x,y\in R$, one has 
\begin{enumerate}[\ \  (V1)]
	\item $\nu(xy) = \nu(x) + \nu(y)$;
	\item $\nu(x+y)\geq \min\{\nu(x),\nu(y) \}$;
	\item $\nu(1) = 0$;
	\item $\nu(0) = \omega$. 
\end{enumerate}
Under the reverse ordering $\leq_\mathrm{op}$, the set $\Gamma^0$ is a upward directed set with a zero (namely, $\omega$) and is also a D-index set (because it is totally ordered). 
Let us define $\fd_\nu: R\times R\to \Gamma^0$ by 
$$\fd_\nu(x,y):= \nu(x-y)\qquad (x,y\in R).$$
Obviously, Condition (V4) implies Condition (D1) in Definition \ref{defn:gen-met}(b). 
On the other hand, Conditions (V1) and (V3) implies that $\nu(-x) = \nu(x)$, and this verifies Condition (D2). 
Using Condition (V2), we know that for every $x,y,z\in R$, we have
$$\fd_\nu(x,z) = \nu(x-z)\leq_\mathrm{op} \max \{\fd_\nu(x,y), \fd_\nu(y,z) \}.$$
Consequently, $\fd_\nu$ is a pseudo ultra $\Gamma^0$-metric. 

Suppose, furthermore, that $R$ is a division ring.
Then for every $x\in R\setminus \{0\}$, we learn from Condition (V1) that 
$$0 = \nu(xx^{-1}) = \nu(x)+\nu(x^{-1}),$$
which implies that $\nu(x)\neq \omega$. 
Thus, in this case, $\big\{(x,y)\in R\times R: \fd_\nu(x,y)=\omega \big\} = \Delta_R$. 
Furthermore, since a totally ordered group can never has a greatest element, we know that $(\Gamma, \leq_\mathrm{op})$ can never has a smallest element. 
Therefore, if $R$ is a division ring, then $\fd_\nu$ is a uniform $\Gamma^0$-metric (see Remark \ref{rem:no-atom}). 
\end{eg}

\medskip

\section*{Acknowledgement}

The author is supported by National Natural Science Foundation of China (11471168) and (11871285).

\medskip

\bibliographystyle{plain}

\end{document}